\documentclass[a4paper,12pt]{article}
\usepackage{a4wide}
\usepackage{amsmath}
\usepackage{amssymb}
\usepackage{amsthm}
\usepackage{latexsym}
\usepackage{graphicx}
\usepackage[english]{babel}
\usepackage{makeidx}
\usepackage[mathlines]{lineno}
\newcommand{\stir}{\genfrac{[}{]}{0pt}{}}

\newtheorem{obs} [subsection]{Remark}
\newtheorem{exm} [subsection]{Example}

\newtheorem{prop}[subsection]{Proposition}

\newtheorem{teor}[subsection]{Theorem}

\newtheorem{cor} [subsection]{Corollary}

\newcommand{\paa}{p_{\mathbf a}}
\newcommand{\Pa}{P_{\mathbf a}}

\def\p{\operatorname{p}}
\def\pp{\operatorname{pp}}
\newcommand{\pps}{\pp^{s}}
\newcommand{\ppso}{\pp^{so}}
\newcommand{\sunca}{\left\lfloor \frac{n+1}{2} \right\rfloor\cdot \left\lfloor \frac{n+2}{2} \right\rfloor}
\newcommand{\soric}{\binom{\left\lfloor \frac{n}{2} \right\rfloor}{2} +\varepsilon(n)}

\begin{document}
\selectlanguage{english}
\frenchspacing

\numberwithin{equation}{section}

\title{A note on the number of plane partitions and $r$-component multipartitions of $n$}
\author{Mircea Cimpoea\c s$^1$ and Alexandra Teodor$^2$}
\date{}

\maketitle

\footnotetext[1]{ \emph{Mircea Cimpoea\c s}, University Politehnica of Bucharest, Faculty of
Applied Sciences, 
Bucharest, 060042, Romania and Simion Stoilow Institute of Mathematics, Research unit 5, P.O.Box 1-764,
Bucharest 014700, Romania, E-mail: mircea.cimpoeas@upb.ro,\;mircea.cimpoeas@imar.ro}
\footnotetext[2]{ \emph{Alexandra Teodor}, University Politehnica of Bucharest, Faculty of
Applied Sciences, 
Bucharest, 060042, E-mail: alexandra.teodor@upb.ro}

\begin{abstract}
Using elementary methods, we prove new formulas for $\pp(n)$, the number of plane partitions of $n$, $\pp_r(n)$, the number
of plane partitions of $n$ with at most $r$ rows, $\pps(n)$, the number of strict plane partitions of $n$ 
and $\ppso(n)$, the number of symmetric plane partitions of $n$. Also, we give new formulas for $P_r(n)$,
the number of $r$-component multipartitions of $n$.

\textbf{Keywords}: Integer partition, Restricted partition function, Plane partition, Multipartition.

\textbf{MSC2010}: 11P81, 11P83.
\end{abstract}

\maketitle
\section{Introduction}

Let $n$ be a positive integer. We denote $[n]=\{1,2,\ldots,n\}$.
A partition of $n$ is a non-increasing sequence $\lambda=(\lambda_1,\ldots,\lambda_m)$ of positive integers such that 
$|\lambda|=\lambda_1+\cdots+\lambda_m=n$. We define
$\p(n)$ as the number of partitions of $n$ and for convenience, we define $p(0) = 1$.
This notion has the following generalization:
A \emph{plane partition} of $n$ is an array $(n_{ij})_{i,j\in [n]}$ of nonnegative integers such that 
$$\sum_{i,j\in [n]}n_{ij}=n\text{ and }n_{ij}\geq n_{i'j'}\text{ for all }i,j,i',j'\in [n]\text{ such that }i\leq i'\text{ and }j\leq j'.$$ 
If $n_{ij}>n_{i(j+1)}$ whenever $n_{ij}\neq 0$, then we shall call such a partition \emph{strict}.
If $n_{ij}=n_{ji}$ for all $i$ and $j$, then the partition is called \emph{symmetric}.

For example, there are $6$ plane partitions of $n=3$, namely:
$$ \begin{matrix} 3 & 0 & 0 \\ 0 & 0 & 0 \\ 0 & 0 & 0 \end{matrix},\;\;\;\;\; 
   \begin{matrix} 2 & 1 & 0 \\ 0 & 0 & 0 \\ 0 & 0 & 0 \end{matrix},\;\;\;\;\;
   \begin{matrix} 2 & 0 & 0 \\ 1 & 0 & 0 \\ 0 & 0 & 0 \end{matrix},\;\;\;\;\; 
   \begin{matrix} 1 & 1 & 1 \\ 0 & 0 & 0 \\ 0 & 0 & 0 \end{matrix},\;\;\;\;\; 
	 \begin{matrix} 1 & 1 & 0 \\ 1 & 0 & 0 \\ 0 & 0 & 0 \end{matrix},\;\;\;\;\;
	 \begin{matrix} 1 & 0 & 0 \\ 1 & 0 & 0 \\ 1 & 0 & 0 \end{matrix}. $$
Note that four of them are strict partitions and two of them are symmetric. Also, three of them have nonzero entries only 
on the first row and five of them have nonzero entries on the first two rows.

We denote $\pp(n)$ the total number of plane partitions of $n$ and, we define $\pp(0)=1$.
The properties of $\pp(n)$ have been extensively studied in literature. As a curiosity,
we mention that the function $\pp(n)$ appears in physics in connection with the
enumeration of small black holes in string theory, see \cite[Appendix E]{blackholes}.

Let $k\geq 1$ be an integer. We denote $\pp_r(n)$, the number of plane partitions with at most $r$ rows, and $\pp_k(0)=1$. 
Note that $\pp_1(n)=\p(n)$ for all $n\geq 0$. Also, if $k\geq n$ then $\pp_k(n)=\pp(n)$.
In the example above, we have $\pp_1(3)=\p(3)=3$, $\pp_2(3)=5$ and $\pp_3(3)=\pp(3)=6$.

We denote $\pps(n)$, the number of strict plane partitions of $n$, and $\ppso(n)$,
the number of strict plane partitions of $n$ with odd parts.
We set $\pps(0)=\ppso(n)=1$. It is well known that $\ppso(n)$ counts also the number of 
symmetric plane partitions of $n$. 

Let $\mathbf a := (a_1, a_2, \ldots , a_r)$ be a sequence of positive integers, $r \geq 1$. The \emph{restricted partition
function} associated to $\mathbf a$ is $\paa : \mathbb N \to \mathbb N$, $\paa(n) :=$ the number of integer solutions $(x_1, \ldots, x_r)$
of $\sum_{i=1}^r a_ix_i = n$ with $x_i \geq 0$. Note that the generating function of $\paa(n)$ is
\begin{equation}\label{gen}
\sum_{n=0}^{\infty}\paa(n)z^n= \frac{1}{(1-z^{a_1})\cdots(1-z^{a_r})},\;|z|<1.
\end{equation}
See \cite[Chapter 5]{andrews} for further details.

Our aim is to provide new formulas for $\pp(n)$, $\pp_r(n)$, $\pp^s(n)$ and $\pp^{so}(n)$ 
using their generating functions and the relation with the restricted partition function $\p_{\mathbf a}(n)$; 
see Proposition \ref{p1}.
In Theorem \ref{formu}, we prove a new formula for $\pp(n)$ in terms
of coefficients of the polynomials 
$$f_s(z)=(1+z+\cdots+z^{\alpha_s})^s,\;\text{ where }\alpha_s=\frac{D_n}{s}-1,$$
where $1\leq s\leq n$ and $D_n$ is the least common multiple of $1,2,\ldots,n$.
Similarly, we provide formulas for $\pp_r(n)$, $\pps(n)$ and $\ppso(n)$ in Theorem \ref{formu2}, Theorem \ref{formu3}
and Theorem \ref{formu4}, respectively. 
Also, using a result from \cite{lucrare} regarding the restricted partition function, we deduce other formulas for 
$\pp(n)$, $\pp_r(n)$, $\pps(n)$ and $\ppso(n)$; see Theorem \ref{tn}.

A $r$-component multipartition of $n$ is a $r$-tuple $\lambda=(\lambda^1,\ldots,\lambda^r)$ of partitions of $n$
such that $|\lambda|=|\lambda^1|+\cdots+|\lambda^r|=n$; see \cite{andrews2}. We denote $P_r(n)$, the number or
$r$-component multipartitions of $n$ and $P_r(0)=1$. In Proposition \ref{p32} we show that 
$$\pp_r(n)=\sum_{\substack{0\leq t_1\leq r-1 \\ \vdots \\ 0\leq t_{r-1}\leq 1}}(-1)^{t_1+\cdots+t_{r-1}}P_r(n-t_1-2t_2-\cdots-(r-1)t_{r-1}),$$
where $P_r(j)=0$ for $j<0$. 

In Theorem \ref{formu5} we prove a new formula for $P_r(n)$ and we deduce from it a new expression
for $\pp_r(n)$; see Corollary \ref{c35}. Also, in Theorem \ref{tn2} we obtain other formula for $P_r(n)$.

\newpage
\section{New formulas for the number of plane partitions}

\begin{itemize}
\item Let $n,k$ be two positive integers.

\item Let $\pp(n)$ be the number of plane partitions of $n$. We define $\pp(0)=1$.
 
\item Let $\pp_r(n)$ be the number of plane partitions of $n$ with at most $r$ rows. We also define $\pp_r(0)=1$.
Note that, if $n\leq r$ then $\pp(n)=\pp_r(n)$.

\item Let $\pps(n)$ be the number of strict plane partitions of $n$. We set $\pps(0)=1$.

\item Let $\ppso(n)$ be the number of strict plane partitions of $n$ with odd parts. We set $\ppso(n)=1$
As it was shown in \cite{gordonv}, $\ppso(n)$ is equal to the number of symmetric plane partitions of $n$.
\end{itemize}

MacMahon \cite{mahon} proved that
\begin{equation}\label{mah1}
\sum_{n=0}^{\infty}\pp(n)z^n = \prod_{n=1}^{\infty} \frac{1}{(1-z^n)^n}\text{ for }|z|<1.
\end{equation}
A refinement of this result is the following
\begin{equation}\label{mah2}
\sum_{n=0}^{\infty}\pp_r(n)z^n = \prod_{n=1}^{\infty} \frac{1}{(1-z^n)^{\min\{n,r\}}}\text{ for }|z|<1,
\end{equation}
see \cite[Equation (10.1)]{andrews}.

Gordon and Houten \cite{gordon} proved that
\begin{equation}\label{gor1}
 \sum_{n=0}^{\infty}\pps(n)z^n = \prod_{n=1}^{\infty} \frac{1}{(1-z^n)^{\left\lfloor (n+1)/2 \right\rfloor}}\text{ for }|z|<1.
\end{equation}
Also, Gordon \cite{gordonv} proved that
\begin{equation}\label{gor2}
\sum_{n=0}^{\infty}\ppso(n)z^n = \prod_{n=1}^{\infty} \frac{1}{(1-z^{2n+1})}\prod_{n=1}^{\infty} \frac{1}{(1-z^{2n})^n} \text{ for }|z|<1.
\end{equation}

Give $n,k$ two positive integers, we denote $n^{[k]}$ the sequence $n,n,\ldots,n$ of length $k$, 
e. g. $2^{[3]}=2,2,2$. We consider the following sequences of integers:
\begin{align*}
& \mathbf n:=(1,2^{[2]},3^{[3]},\ldots,n^{[n]}),\\
& \mathbf n_r:=(1,2^{[\min\{2,r\}]},3^{[\min\{3,r\}]},\ldots,n^{[\min\{n,r\}]}),\\
& \mathbf{n}^s:=(1,2^{[1]},3^{[2]},4^{[2]},\ldots,n^{[\lfloor (n+1)/2 \rfloor]})\text{ and }\\
& \mathbf{n}^{so}:=(1,2^{[1]},3^{[1]},4^{[2]},\ldots,n^{[\lambda(n)]}),
\end{align*}
where $\lambda(n)=\begin{cases} 1,&n\text{ is odd} \\ \frac{n}{2},&n\text{ is even} \end{cases}$.

\begin{prop}\label{p1}
Let $n,k$ be two positive integers. We have that:
\begin{enumerate}
\item[(1)] $\pp(n)=p_{\mathbf n}(n)$ for all $n\geq 0$.
\item[(2)] $\pp_r(n)=p_{\mathbf n_r}(n)$ for all $n\geq 0$.
\item[(3)] $\pps(n)=p_{\mathbf n^s}(n)$ for all $n\geq 0$.
\item[(4)] $\ppso(n)=p_{\mathbf n^{so}}(n)$ for all $n\geq 0$.
\end{enumerate}
\end{prop}

\begin{proof}
(1) From \eqref{gen}, it follows that
\begin{equation}\label{gogu1}
\sum_{m=0}^{\infty}p_{\mathbf n}(m)z^m=\prod_{j=1}^n \frac{1}{(1-z^j)^j} = (1+z+z^2+\cdots)(1+z^2+z^4+\cdots)\cdots (1+z^n+z^{2n}+\cdots).
\end{equation}
On the other hand, from \eqref{mah1}, we have
\begin{equation}\label{gogu2}
\sum_{m=0}^{\infty}\pp(m)z^m = (1+z+z^2+\cdots)(1+z^2+z^4+\cdots)\cdots (1+z^n+z^{2n}+\cdots)(1+z^{n+1}+\cdots)\cdots.
\end{equation}
Comparing \eqref{gogu1} and \eqref{gogu2}, it follows that 
$$\pp(m) = p_{\mathbf n}(m)\text{ for all }0\leq m\leq n,$$
thus, in particular, $\pp(n)=p_{\mathbf n}(n)$, as required.

(2,3,4) The proof is similar to the proof of (1), so we omit it. 
\end{proof}

Let $D_n$ be the least common multiple of $1,2,\ldots,n$.
For all $1\leq s\leq n$, we define
$$f_s(z):=(1+z+\cdots+z^{\alpha_s})^s,\;\text{ where }\alpha_s=\frac{D_n}{s}-1.$$
We write 
\begin{equation}\label{ek}
f_s(z)=f_{s,0}+f_{s,1}z+\cdots+f_{s,s\alpha_s}z^{s\alpha_s}.
\end{equation}
From the definition of $f_s(z)$, we note that 
\begin{equation}\label{fjl}
f_{s,\ell} =\# \{(j_1,\ldots,j_s)\;:\;j_1+\cdots+j_s=\ell,\text{ where }0\leq j_t\leq \alpha_s,\; 1\leq t\leq s\}.
\end{equation}
Note that the polynomial $f_s(z)$ is reciprocal, that is,
$$x^{s\alpha_s}f_s\left(\frac{1}{x}\right) = f_s(x), $$
so that $f_{s,i}=f_{s,s\alpha_s-i}$ for all $i$. On the other hand, we have
$$f_s(z)=\left(\frac{1-z^{\frac{D_n}{s}}}{1-z}\right)^s = (1-z^{\frac{D_n}{s}})^s(1-z)^{-s}.$$
Using the binomial expansion of $(1-z)^{-s}$, it follows that
\begin{equation}\label{yek}
f_s(z)=\sum_{i=0}^s (-1)^i \binom{s}{i} z^{i\frac{D_n}{s}}\times \sum_{j=0}^{\infty} \binom{j+s-1}{j}z^j.
\end{equation}
Comparing \eqref{ek} with \eqref{yek} we deduce that:
\begin{equation}\label{duoh}
f_{s,\ell}=\sum_{i,j\geq 0\text{ with }\frac{iD_n}{s}+j=\ell} (-1)^i\binom{s}{i}\binom{j+s-1}{j},
\end{equation}
for all $0\leq \ell\leq s\alpha_s$.

\begin{teor}\label{formu}
Let $n\geq 3$ be an integer. We have that 
$$\pp(n) = \sum_{(\ell_1,\ldots,\ell_n)\in \mathbf{A}_n} 
\prod_{s=2}^n \sum_{\substack{i_s,j_s\geq 0\text{ with }\\ \frac{i_sD_n}{s}+j_s=\ell_s}} (-1)^{i_s}\binom{s}{i_s}\binom{j_s+s-1}{j_s},$$
where $\mathbf{A}_n:=\{(\ell_1,\ldots,\ell_n)\in\mathbb N^n\;:\; \ell_1+2\ell_2+\cdots+n\ell_n = n\}$.
\end{teor}

\begin{proof}
From Proposition \ref{p1}(1) it follows that
\begin{equation}\label{pepene}
\pp(n)=\#\{(j_1,\ldots,j_{\binom{n+1}{2}})\in\mathbb N^{\binom{n+1}{2}}\;:\; j_1+2j_2+2j_3+\cdots+nj_{\binom{n}{2}+1} + \cdots + 
nj_{\binom{n+1}{2}} = n\}.
\end{equation}
We denote $\ell_1=j_1,\; \ell_2=j_2+j_3,\ldots,\ell_n=j_{\binom{n}{2}+1} + \cdots + j_{\binom{n+1}{2}}$.

Since $n\geq 3$, it follows that 
$$\ell_s \leq \frac{n}{s} \leq \frac{D_n}{s}-1\text{ for all }2\leq s\leq n.$$
Hence, from \eqref{fjl} and \eqref{pepene} it follows that
\begin{equation}\label{alexuta}
\pp(n)=\sum_{(\ell_1,\ldots,\ell_n)\in \mathbf{A}_n}\prod_{s=2}^{n}f_{s,\ell_s}.
\end{equation}
Therefore, from \eqref{alexuta} and \eqref{duoh} we get the required formula.
\end{proof}

\begin{obs}\rm
Note that $p(n)=\#\mathbf A_n$ for all $n\geq 3$.
\end{obs}

\begin{exm}\rm
If we take $n=3$ in Theorem \ref{formu}, we have $\mathbf{A}_3=\{(3,0,0),(1,1,0),(0,0,1)\}$ and thus
$$ \pp(3)=\sum_{(\ell_1,\ell_2,\ell_3)\in\mathbf A_3} \prod_{s=2}^3 \sum_{\frac{6i_s}{s}+j_s=\ell_s} (-1)^{i_s}\binom{j_s+s-1}{j_s} = 
1 + \binom{1}{1}\binom{2}{1} + \binom{3}{1}=6.$$
\end{exm}

\begin{teor}\label{formu2}
Let $n > r\geq 2$ be two positive integers. We have that 
$$\pp_r(n) = \sum_{(\ell_1,\ldots,\ell_n)\in \mathbf{A}_{n}}
\prod_{s=2}^n \sum_{\substack{i_s,j_s\geq 0\text{ with }\\ \frac{i_sD_n}{\min\{s,r\}}+j_s=\ell_s}} (-1)^{i_s}\binom{\min\{s,r\}}{i_s}\binom{j_s+\min\{s,r\}-1}{j_s}.$$
\end{teor}

\begin{proof}
From Proposition \ref{p1}(2) it follows that
$$
\pp_r(n)=\#\{(j_1,\ldots,j_{rn -\binom{r}{2}})\in\mathbb N^{rn -\binom{r}{2}}\;:\; j_1+2j_2+2j_3+\cdots+rj_{\binom{r}{2}+1} + \cdots + rj_{\binom{r+1}{2}} + $$
\begin{equation}\label{pepene2}
+ \cdots + n j_{rn -\binom{r+1}{2}+1} + \cdots + 
					n j_{rn -\binom{r}{2}}  = n\}.
\end{equation}
We denote $\ell_1=j_1,\; \ell_2=j_2+j_3,\ldots,\ell_n=j_{rn -\binom{r+1}{2}+1} + \cdots + j_{rn -\binom{r}{2}}$.

Since $n\geq 3$ and $r\geq 2$ it follows that 
$$\ell_s \leq \frac{n}{s} \leq \frac{D_n}{s}-1 \leq \frac{D_n}{\min\{r,s\}}-1\text{ for all }2\leq s\leq n.$$
Hence, from \eqref{pepene2} and \eqref{fjl} it follows that:
\begin{equation}\label{alexuta2}
\pp_r(n)=\sum_{(\ell_1,\ldots,\ell_n)\in \mathbf A_n} \prod_{s=2}^n f_{\min\{r,s\},\ell_{s}}
\end{equation}
Therefore, from \eqref{alexuta2} and \eqref{duoh} we get the required formula.
\end{proof}

\begin{exm}\rm
Let $n=3$ and $r=2$. Since $\mathbf{A}_3=\{(3,0,0),(1,1,0),(0,0,1)\}$, according to Theorem \ref{formu2}, the number of plane partitions of 
$3$ with at most $2$ rows is:
$$\pp_2(3)= \sum_{(\ell_1,\ell_2,\ell_3)\in\mathbf A_3} \prod_{s=2}^3 \sum_{i_s,j_s\geq 0,\;3i_s+j_s=\ell_s}\binom{2}{i_s}\binom{j_s+1}{j_s}= 1+
\binom{2}{0}\binom{2}{1}+\binom{2}{0}\binom{2}{1}=5.$$
\end{exm}

\begin{teor}\label{formu3}
Let $n\geq 3$ be an integer. We have that
$$ \pps(n)=\sum_{(\ell_1,\ldots,\ell_n)\in \mathbf A_n} \prod_{s=3}^n \sum_{\substack{i_s,j_s\geq 0\text{ with }\\ 
\frac{i_sD_n}{\lfloor \frac{s+1}{2} \rfloor}+j_s=\ell_s}} (-1)^{i_s}\binom{\lfloor \frac{s+1}{2} \rfloor}{i_s}
\binom{j_s+\lfloor \frac{s+1}{2} \rfloor-1}{j_s}.$$
\end{teor}

\begin{proof}
Assume $n=2p$. From Proposition \ref{p1}(3) it follows that
\begin{equation}\label{pepene3}
\pps(n) = \# \{(j_1,\ldots,j_{p^2+p})\;:\;j_1+2j_2+3j_3+3j_4+\ldots+nj_{p^2+1}+\cdots+nj_{p^2+p}=n\}
\end{equation}						
We denote $\ell_1=j_1,\; \ell_2=j_2,\; \ell_3=j_3+j_4,\ldots,\ell_n=j_{p^2+1} + \cdots + j_{p^2+p}$.

Since $n\geq 3$ and $k\geq 2$ it follows that 
$$\ell_s \leq \frac{n}{s} \leq \frac{D_n}{s}-1 \leq \frac{D_n}{\lfloor \frac{s+1}{2} \rfloor}-1\text{ for all }1\leq s\leq n.$$
Hence, from \eqref{pepene3} and \eqref{fjl} it follows that:
\begin{equation}\label{alexuta3}
\pps(n)=\sum_{(\ell_1,\ldots,\ell_n)\in \mathbf A_n} \prod_{s=3}^n f_{\lfloor \frac{s+1}{2} \rfloor,\ell_{s}}
\end{equation}	
Therefore, from \eqref{alexuta3} and \eqref{duoh} we get the required formula. The case $n=2p+1$ is similar.
\end{proof}

\begin{exm}\rm
Let $n=3$. Since $\mathbf{A}_3=\{(3,0,0),(1,1,0),(0,0,1)\}$, according to Theorem \ref{formu3}, the number of strict plane partitions
of $3$ is:
$$\pps(3)=\sum_{(\ell_1,\ell_2,\ell_3)\in\mathbf A_3} \sum_{i_3,j_3>0,\;3i_3+j_3=\ell_3}(-1)^{i_3}\binom{2}{i_3}\binom{j_3+1}{j_3}=1+1+\binom{2}{0}\binom{2}{1}=4.$$
\end{exm}

\begin{teor}\label{formu4}
Let $n\geq 3$ be a positive integer. We have that
$$ \ppso(n)=\sum_{(\ell_1,\ldots,\ell_n)\in \mathbf A_n} \prod_{s=2}^{\lfloor \frac{n}{2} \rfloor} 
\sum_{\substack{i_s,j_s\geq 0\text{ with }\\ 
\frac{i_sD_n}{s}+j_s=\ell_{2s}}} (-1)^{i_s}\binom{s}{i_s}
\binom{j_s+s-1}{j_s}.$$
\end{teor}

\begin{proof}
Assume $n=2p$. From Proposition \ref{p1}(4) it follows that
$$\pps(n) = \# \{(j_1,\ldots,j_{p^2+p})\;:\;j_1+2j_2+3j_3+4j_4+4j_5+5j_6+\ldots+(n-1)j_{\binom{p}{2}} +$$
\begin{equation}\label{pepene4}
 + nj_{\binom{p}{2}+1}+\cdots+nj_{\binom{p+1}{2}} =n\}.
\end{equation}						
We denote $\ell_1=j_1,\; \ell_2=j_2,\; \ell_3=j_3, \ldots,\;\ell_{n-1}=j_{\binom{p}{2}},\; \ell_n=j_{\binom{p}{2}+1} + \cdots + j_{\binom{p+1}{2}}$.

Since $n\geq 3$ and $k\geq 2$ it follows that 
$$\ell_{2s} \leq \frac{n}{2s} \leq \frac{D_n}{2s}-1 \leq \frac{D_n}{s}-1\text{ for all }2\leq s\leq p.$$
Hence, from \eqref{pepene3} and \eqref{fjl} it follows that:
\begin{equation}\label{alexuta4}
\pps(n)=\sum_{(\ell_1,\ldots,\ell_n)\in \mathbf A_n} \prod_{s=2}^{\lfloor \frac{n}{2} \rfloor} f_{s,\ell_{2s}}
\end{equation}	
Therefore, from \eqref{alexuta4} and \eqref{duoh} we get the required formula. The case $n=2p+1$ is similar.
\end{proof}

\begin{exm}\rm
Let $n=3$. Since $\mathbf{A}_3=\{(3,0,0),(1,1,0),(0,0,1)\}$ and $\lfloor \frac{3}{2} \rfloor=1$, according to Theorem \ref{formu4}, the number of
symmetric plane partition of $3$ is:
$$\ppso(3)=\p(3)=\#\mathbf{A}_3=3.$$
\end{exm}

The unsigned Stirling numbers $\stir{r}{k}$'s are defined by
\begin{equation}
 \binom{n+r-1}{r-1}=\frac{1}{n(r-1)!}n^{(r)}=\frac{1}{(r-1)!}\left(\stir{r}{r}n^{r-1} + \cdots \stir{r}{2}n + \stir{r}{1}\right).
\end{equation}
We recall the following result:

\begin{teor}(see \cite[Theorem 2.8(2)]{lucrare} and \cite{cori})\label{teora}
Let $\mathbf a=(a_1,\ldots,a_r)$ be sequence of positive integers and let $D$ be the least common multiple of $a_1,\ldots,a_r$. We have that
$$\paa(n) = \frac{1}{(r-1)!} \sum_{m=0}^{r-1}    \sum_{\substack{0\leq j_1\leq \frac{D}{a_1}-1,\ldots, 0\leq j_r\leq \frac{D}{a_r}-1 \\ 
a_1j_1+\cdots+a_rj_r \equiv n (\bmod D)}} 
\sum_{k=m}^{r-1} \stir{r}{k+1} (-1)^{k-m} \binom{k}{m} D^{-k} (a_1j_1 + \cdots + a_rj_r)^{k-m}  n^m.$$
\end{teor}

Using Theorem \ref{teora} we are able to obtain new formulas for $\pp(n)$, $\pp_k(n)$, $\pps(n)$ and $\ppso(n)$,
where $n\geq 3$ is an integer:

\begin{teor}\label{tn}
\begin{enumerate}
\item[(1)] For $n\geq 3$ we have
\begin{align*}
 & \pp(n) = \frac{1}{\left(\binom{n+1}{2}-1\right)!} \sum_{m=0}^{\binom{n+1}{2}-1} 
             \sum_{\substack{0\leq \ell_1\leq D_n-1,\ldots,0 \leq \ell_n\leq D_n-n \\ \ell_1+2\ell_2+\cdots+n\ell_n\equiv n(\bmod\;D_n)}} 
  \prod_{s=2}^n \sum_{\substack{i_s,j_s\geq 0\text{ with }\\ \frac{i_sD_n}{s}+j_s=\ell_s}} (-1)^{i_s} \times  \\
 & \times \binom{s}{i_s}\binom{j_s+s-1}{j_s} \sum_{k=m}^{\binom{n+1}{2}-1} \stir{\binom{n+1}{2}}{k+1} (-1)^{k-m} \binom{k}{m} D_n^{-k} (\ell_1+2\ell_2+\cdots+n\ell_n)^{k-m} n^m.
\end{align*}
\item[(2)] For $n\geq 3$ and $2\leq r\leq n-1$ we have
\begin{align*}
 & \pp_r(n) = \frac{1}{\left(nr - \binom{r}{2}-1\right)!} \sum_{m=0}^{nr - \binom{r}{2}-1} 
             \sum_{\substack{0\leq \ell_1\leq D_n-\min\{1,r\},\ldots,0 \leq \ell_n\leq D_n-\min\{n,r\} \\ \ell_1+2\ell_2+\cdots+n\ell_n\equiv n(\bmod\;D_n)}} 
   \times  \\
 & \times \prod_{s=2}^n \sum_{\substack{i_s,j_s\geq 0\text{ with }\\ \frac{i_sD_n}{\min\{s,r\}}+j_s=\ell_s}} (-1)^{i_s}\binom{\min\{s,r\}}{i_s}\binom{j_s+\min\{s,r\}-1}{j_s}	\times \\
 & \times \sum_{k=m}^{nr - \binom{r}{2}-1} \stir{nr - \binom{r}{2}}{k+1} (-1)^{k-m} \binom{k}{m} D_n^{-k} (\ell_1+2\ell_2+\cdots+n\ell_n)^{k-m} n^m.
\end{align*}

\item[(3)] For $n\geq 3$ we have
\begin{align*}
 & \pps(n) = \frac{1}{\left(\sunca-1\right)!} \sum_{m=0}^{\sunca-1} 
             \sum_{\substack{0\leq \ell_1\leq D_n-1,\ldots,0 \leq \ell_n\leq D_n-\left\lfloor \frac{n+1}{2} \right\rfloor
						\\ \ell_1+2\ell_2+\cdots+n\ell_n\equiv n(\bmod\;D_n)}} \times  \\
 & \times \prod_{s=3}^n \sum_{\substack{i_s,j_s\geq 0\text{ with }\\ 
\frac{i_sD_n}{\lfloor \frac{s+1}{2} \rfloor}+j_s=\ell_s}} (-1)^{i_s}\binom{\lfloor \frac{s+1}{2} \rfloor}{i_s}
\binom{j_s+\lfloor \frac{s+1}{2} \rfloor-1}{j_s}	\times \\
 & \times \sum_{k=m}^{\sunca-1} \stir{\sunca}{k+1} (-1)^{k-m} \binom{k}{m} D_n^{-k} (\ell_1+2\ell_2+\cdots+n\ell_n)^{k-m} n^m.
\end{align*}

\item[(4)] For $n\geq 3$ we have
\begin{align*}
 & \ppso(n) = \frac{1}{\left(\soric-1\right)!} \sum_{m=0}^{\soric-1} 
             \sum_{\substack{0\leq \ell_1\leq D_n-\lambda(1),\ldots,0 \leq \ell_n\leq D_n-\lambda(n) \\ \ell_1+2\ell_2+\cdots+n\ell_n\equiv n(\bmod\;D_n)}} 
   \times  \\
 & \times \prod_{s=2}^{\lfloor \frac{n}{2} \rfloor} 
\sum_{\substack{i_s,j_s\geq 0\text{ with }\\ 
\frac{i_sD_n}{s}+j_s=\ell_{2s}}} (-1)^{i_s}\binom{s}{i_s}
\binom{j_s+s-1}{j_s}	\times \\
 & \times \sum_{k=m}^{\soric-1} \stir{\soric}{k+1} (-1)^{k-m} \binom{k}{m} D_n^{-k} (\ell_1+2\ell_2+\cdots+n\ell_n)^{k-m} n^m,
\end{align*}
where $\lambda(n)=\begin{cases} 1,&n\text{ is odd}\\ \frac{n}{2},&n\text{ is even} \end{cases}$ and $\varepsilon(n)=\begin{cases} 1,&n\text{ is odd}\\0,&n\text{ is even}\end{cases}$.
\end{enumerate}
\end{teor}

\begin{proof}
(1) Note that the length of the sequence $\mathbf n=(1,2^{[2]},\ldots,n^{[n]})$ is $\binom{n+1}{2}$.
From Proposition \ref{p1}(1) and Theorem \ref{teora} it follows that
$$ \pp(n) = \p_{\mathbf n}(n) = \frac{1}{\left(\binom{n+1}{2}-1\right)!} \sum_{m=0}^{\binom{n+1}{2}-1} \sum_{(j_1,\ldots,j_{\binom{n+1}{2}})\in \mathbf C_n} \sum_{k=m}^{\binom{n+1}{2}-1} \stir{\binom{n+1}{2}}{k+1}\times $$
\begin{equation}\label{starc}
\times (-1)^{k-m} \binom{k}{m} D_n^{-k} (j_1 + 2j_2 + 2j_3 + \cdots + nj_{\binom{n}{2}+1} + \cdots + nj_{\binom{n+1}{2}})^{k-m}  n^m,
\end{equation}
where $\mathbf C_n =\{ (j_1,\ldots,j_{\binom{n+1}{2}})\;: 0\leq j_1\leq D_n-1, 0\leq j_2\leq \frac{D_n}{2}-1, 0\leq j_3\leq \frac{D_n}{2}-1,\linebreak \ldots ,
                       0\leq j_{\binom{n}{2}+1}\leq \frac{D_n}{n}-1,\ldots,0\leq j_{\binom{n+1}{2}}\leq \frac{D_n}{n}-1\text{ such that } 
											j_1+2j_2+2j_3+\cdots+nj_{\binom{n}{2}+1} + \cdots + nj_{\binom{n+1}{2}} \equiv n(\bmod\;D_n)\}$.
											
We let $\ell_1=j_1,\; \ell_2=j_2+j_3,\; \ldots, \ell_n=j_{\binom{n}{2}+1}+\cdots+j_{\binom{n+1}{2}}$.

Note that if $(j_1,j_2,\ldots,j_{\binom{n+1}{2}})\in \mathbf C_n$ then $0\leq \ell_t \leq D_n-t$ for all $1\leq t\leq n$ and, moreover,
$\ell_1+2\ell_2+\cdots+n\ell_n \equiv n(\bmod\;D_n)$.
From \eqref{starc}, using a similar argument as in the proof of Theorem \ref{formu}, we get the required result.

(2) Note that the length of the sequence $\mathbf n_r:=(1,2^{[\min\{2,r\}]},3^{[\min\{3,r\}]},\ldots,n^{[\min\{n,r\}]})$ is
    $\binom{r+1}{2} + r(n-r) = nr - \binom{r}{2}$.		
		The rest of the proof is similar to the proof of (1), using Proposition \ref{p1}(2), Theorem \ref{formu2} and Theorem \ref{teora}.
		
(3) Note that the length of the sequence $\mathbf{n}^s:=(1,2^{[1]},3^{[2]},4^{[2]},\ldots,n^{[\lfloor (n+1)/2 \rfloor]})$ is $p^2+p$ if
    $n=2p$ and $(p+1)^2$ if $n=2p+1$. Hence, in both cases, the length if $\sunca$.
    The rest of the proof is similar to the proof of (1), using Proposition \ref{p1}(3), Theorem \ref{formu3} and Theorem \ref{teora}.
		
(4) Note that the length of the sequence $\mathbf{n}^{so}:=(1,2^{[1]},3^{[1]},4^{[2]},\ldots,n^{[\lambda(n)]})$ is 
    $\binom{p}{2}$ if $n=2p$ and $\binom{p}{2}+1$ if $n=2p+1$. Hence, in both cases, the length if $\soric$.
		The rest of the proof is similar to the proof of (1), using Proposition \ref{p1}(4), Theorem \ref{formu4} and Theorem \ref{teora}.
\end{proof}

\section{New formulas for the number of multipartitions with $r$ components}

If we denote $P_r(n)$ the number of $r$-component multipartitions of $n$, we have that
\begin{equation}\label{koko}
\sum_{n=0}^{\infty}P_r(n)z^n = \prod_{n=1}^{\infty}\frac{1}{(1-z^n)^r},
\end{equation}
where $P_r(0)=0$ by convention; see \cite{andrews2}. We consider the sequence
$$\mathbf{n}^r =(1^{[r]},2^{[r]},\ldots,n^{[r]}).$$

\begin{prop}\label{p2}
With the above notations, we have $P_r(n)=\p_{\mathbf{n}^r}(n)$ for all $n\geq 0$.
\end{prop}

\begin{proof}
It follows from \eqref{gen} and \eqref{koko}, using a similar argument as in the proof of Proposition \ref{p1}(1).
\end{proof}

\begin{teor}\label{formu5}
Let $n\geq 4$ and $r\geq 2$ be two positive integers such that $n>r$. We have that 
$$P_r(n) = \sum_{(\ell_1,\ldots,\ell_n)\in \mathbf{A}_{n}}
\prod_{s=1}^n \sum_{\substack{i_s,j_s\geq 0\text{ with }\\ \frac{i_sD_n}{r}+j_s=\ell_s}} (-1)^{i_s}\binom{r}{i_s}\binom{j_s+r-1}{j_s}.$$
\end{teor}

\begin{proof}
From Proposition \ref{p2} it follows that
\begin{equation}\label{pepene5}
P_r(n) = \# \{(j_1,\ldots,j_{nr})\;:\;j_1+\cdots+j_r+ 2j_{r+1}+\cdots +2j_{2r}+\cdots+nj_{nr-r+1}+\cdots+nj_{nr}=n\}
\end{equation}						
We denote $\ell_1=j_1+\cdots+j_r,\; \ell_2=j_{r+1}+\cdots+j_{2r},\ldots, \ell_n=j_{nr-r+1}+\cdots+j_{nr}$.

Since $n\geq 4$ and $n>r\geq 2$ it follows that 
$$\ell_s \leq \frac{n}{s} \leq \frac{D_n}{r}-1 \text{ for all }1\leq s\leq n.$$
Hence, from \eqref{pepene5} and \eqref{fjl} it follows that:
\begin{equation}\label{alexuta5}
P_r(n)=\sum_{(\ell_1,\ldots,\ell_n)\in \mathbf A_n} \prod_{s=1}^n f_{r,\ell_{s}}
\end{equation}	
Therefore, from \eqref{alexuta5} and \eqref{duoh} we get the required formula. 
\end{proof}

\begin{exm}\rm
Let $n=4$ and $r=2$.
Since $\mathbf A_4=\{(4,0,0,0),(2,1,0,0),(1,0,1,0),(0,2,0,0)$, $(0,0,0,1)\}$ and $f_{2,\ell}=\ell+1$ for $0\leq \ell\leq 4$,
from Theorem \ref{formu5} we have that
$$P_2(4)= 5+3\cdot 2+2\cdot 2 +3+2=20.$$
\end{exm}

Let $\pp_r(n)$ be the number of plane partitions of $n$ with at most $r$ rows. From \eqref{mah2} and \eqref{koko}
it follows that
$$ \sum_{n=0}^{\infty} \pp_r(n)z^n = \prod_{j=1}^{r-1}(1-z^j)^{r-j} \sum_{n=0}^{\infty} P_r(n)z^n = $$
\begin{equation}\label{azorel}
= \sum_{\substack{0\leq t_1\leq r-1 \\ 0 \leq t_2 \leq r-2 \\ \vdots \\ 0\leq t_{r-1}\leq 1}}(-1)^{t_1+\cdots+t_{r-1}} \sum_{n=0}^{\infty}P_r(n-t_1-2t_2-\cdots-(r-1)t_{r-1})z^{n},
\end{equation}
where $P_r(j)=0$ for $j<0$. As a direct consequence of \eqref{azorel}, we get the following result:

\begin{prop}\label{p32}
For all $n\geq 0$ we have that
$$\pp_r(n)=\sum_{\substack{0\leq t_1\leq r-1 \\ \vdots \\ 0\leq t_{r-1}\leq 1}}(-1)^{t_1+\cdots+t_{r-1}}P_r(n-t_1-2t_2-\cdots-(r-1)t_{r-1}).$$
\end{prop}

As a consequence of Proposition \ref{p32} and Theorem \ref{formu5}, we get:

\begin{cor}\label{c35}
For all $n\geq 4$ and $n > r\geq 2$, we have that:
\begin{align*}
& \pp_r(n) = \sum_{\substack{0\leq t_1\leq r-1 \\ \vdots \\ 0\leq t_{r-1}\leq 1 \\  n < 4+t_1+2t_2+\cdots+(r-1)t_{r-1}}}(-1)^{t_1+\cdots+t_{r-1}}P_r(n-t_1-2t_2-\cdots-(r-1)t_{r-1}) + \\
& \sum_{\substack{0\leq t_1\leq r-1 \\ \vdots \\ 0\leq t_{r-1}\leq 1 \\  n\geq 4+t_1+2t_2+\cdots+(r-1)t_{r-1}}}(-1)^{t_1+\cdots+t_{r-1}}
 \sum_{\substack{ (\ell_1,\ldots,\ell_n)\in \\ \mathbf{A}_{n-t_1-2t_2-\cdots-(r-1)t_{r-1}}}}
\prod_{s=1}^n \sum_{\substack{i_s,j_s\geq 0\text{ with }\\ \frac{i_sD_n}{r}+j_s=\ell_s}} (-1)^{i_s}\binom{r}{i_s}\binom{j_s+r-1}{j_s}. 
\end{align*}
\end{cor}

\begin{teor}\label{tn2}
Let $n\geq 4$ and $r\geq 2$ be two positive integers such that $n\geq k$. 
We have that 
\begin{align*}
& P_r(n) = \frac{1}{\left(nr-1\right)!} \sum_{m=0}^{nr - 1}
             \sum_{\substack{0\leq \ell_1\leq D_n-1,\ldots,0 \leq \ell_n\leq D_n-n \\ \ell_1+2\ell_2+\cdots+n\ell_n\equiv n(\bmod\;D_n)}} \prod_{s=1}^n \sum_{\substack{i_s,j_s\geq 0\text{ with }\\ \frac{i_sD_n}{r}+j_s=\ell_s}} (-1)^{i_s}\binom{r}{i_s}\binom{j_s+k-1}{j_s} \times \\
& \times \sum_{k=m}^{nr-1} \stir{nr}{k+1} (-1)^{k-m} \binom{k}{m} D_n^{-k} (\ell_1+2\ell_2+\cdots+n\ell_n)^{k-m} n^m.
\end{align*}
\end{teor}

\begin{proof}
From Proposition \ref{p2} and Theorem \ref{teora} it follows that 
$$ P_r(n) = \frac{1}{\left(nr-1\right)!} \sum_{m=0}^{nr - 1} \sum_{\substack{0\leq j_1\leq D_n-1,\ldots,0 \leq j_r \leq D-1 \\ \vdots \\ 
 0\leq j_{nr-r+1} \leq \frac{D_n}{n}-1,\ldots, 0\leq j_{nr} \leq \frac{D_n}{n}-1 \\ j_1+\cdots+j_r+\cdots+nj_{nr-r+1} + \cdots + nj_{nr} \equiv n(\bmod\;D_n)}}
\sum_{k=m}^{nr-1} \stir{nr}{k+1} (-1)^{k-m} \times $$
\begin{equation}\label{curcan}
\times \binom{k}{m} D_n^{-k} (j_1+ \cdots +j_r+ \cdots + nj_{rn-r+1}+\cdots +nj_{rn})^{k-m} n^m.
\end{equation}
We let $\ell_1=j_1+\cdots+j_r, \ldots, \ell_n=j_{nr-r+1}+\cdots+j_{rn}$. It is easy to see that $0\leq\ell_t\leq D_n-t$ for $1\leq t\leq n$.
Therefore, from \eqref{curcan}, using a similar argument as in the proof of Theorem \ref{formu5}, it follows the required formula.
\end{proof}

\section{Conclusions}

We proved new formulas for $\pp(n)$, the number of plane partitions of $n$, $\pp_k(n)$, the number of plane partitions with at most $k$ rows of $n$, $\pps(n)$, the number o strict plane partitions of $n$, $\ppso(n)$, the number of symmetric plane partitions, and
$P_k(n)$, the number of multipartitions with $k$ components.

Further investigations include the study of restricted plane partitions with at most $r$ rows and $c$ columns or, more generally, with a given shape.

\subsection*{Acknowledgments} 

The first author (Mircea Cimpoea\c s) was supported by a grant of the Ministry of Research, Innovation and Digitization, CNCS - UEFISCDI, 
project number PN-III-P1-1.1-TE-2021-1633, within PNCDI III.

\pagebreak






\end{document}